\documentclass[11pt]{amsart}
\usepackage[utf8]{inputenc}
\usepackage[english]{babel}
\usepackage[a4paper,top=2.9cm,bottom=2.9cm,left=3cm,right=3cm,marginparwidth=1.75cm]{geometry}
\usepackage[backend=biber, doi=false, url=false, isbn=false, style=trad-abbrv, giveninits]{biblatex} %
\usepackage{csquotes}

\usepackage{amsmath}
\usepackage{enumitem}
\usepackage{graphicx}
\usepackage[colorlinks=true, allcolors=blue]{hyperref}
\usepackage{amsfonts,bm}
\usepackage{amsbsy,amsthm,mathtools}
\usepackage{stmaryrd} %
\usepackage{latexsym,amsmath,amssymb,verbatim}
\usepackage{tikz}
\usepackage{ytableau} %

\theoremstyle{plain}
   
   \newtheorem{theorem}{Theorem}[section]

   \newtheorem{lemma}[theorem]{Lemma}
   \newtheorem{corollary}[theorem]{Corollary}
   
   \newtheorem{problem}[theorem]{Problem}
   \newtheorem{observation}[theorem]{Observation}
   
\theoremstyle{definition}
   
   \newtheorem{defn}[theorem]{Definition}
   \newtheorem{example}[theorem]{Example}

   \newtheorem{remark}[theorem]{Remark}

\newcommand\suchthat{\:\ifnum\currentgrouptype=16\middle\fi|\:}

\DeclareMathOperator{\wt}{wt}

\newcommand{\SSYT}{{\operatorname{SSYT}}}
\newcommand{\rSSYT}{{\operatorname{rSSYT}}}

\newcommand{\Leg}{{\operatorname{Leg}}}
\newcommand{\Num}{{\operatorname{Num}}}
\newcommand{\Den}{{\operatorname{Den}}}
\newcommand{\LTplus}{L_{T}^{+}}
\newcommand{\Lplus}{L}
\newcommand{\HLS}{{\operatorname{HLS}}}
\newcommand{\WO}{{\operatorname{WO}}}
\newcommand{\Mult}{{\operatorname{Mult}}}
\newcommand{\supp}{{\operatorname{supp}}}
\newcommand{\lerl}{\mathrel{\le_{\mathrm{rl}}}}
\newcommand{\ltrl}{\mathrel{<_{\mathrm{rl}}}}
\newcommand{\letb}{\mathrel{\le_{t}}}
\newcommand{\lttb}{\mathrel{<_{t}}}
\newcommand{\tcol}{\mathrel{\sim_{\mathrm{set}}}}

\newcommand{\bfn}{{\underline{n}}}
\newcommand{\bfr}{{\underline{r}}}
\newcommand{\bfX}{{\bm{X}}}
\newcommand{\bfY}{{\bm{Y}}}
\newcommand{\bfzero}{{\bm{0}}}
\newcommand{\bfone}{{\bm{1}}}

\newcommand{\Pt}{P}%
\newcommand{\OC}{\Delta} %
\newcommand{\Sc}{S^c}%

\def\NN{{\mathbb N}}
\def\ZZ{{\mathbb Z}}
\def\QQ{{\mathbb Q}}

\makeatletter
\@namedef{subjclassname@2020}{2020 MSC}
\makeatother

\title{Reciprocity of Skew Hall--Littlewood--Schubert Series}%

\author{Ron M.\ Adin}
\address{Department of Mathematics, Bar-Ilan University, 
	Ramat-Gan 5290002, Israel}
\email{radin@math.biu.ac.il}

\author{Tomer Bauer}
\address{Department of Mathematics, Bar-Ilan University, 
	Ramat-Gan 5290002, Israel}
\email{mathzeta2@gmail.com}



\date{March 31, 2026}
\renewbibmacro{in:}{}

\begin{document}

\begin{abstract}
    Carnevale, Schein and Voll proved self-reciprocity of the generalized Igusa functions, and Maglione and Voll did the same for the Hall--Littlewood--Schubert series.
    We introduce a simultaneous generalization and refinement of these two rational functions, 
    and prove that it satisfies a self-reciprocity property. 
    This answers a problem posed by Maglione and Voll.
    Our method of proof is elementary, avoiding the use of $p$-adic integration.
\end{abstract}

\maketitle

\tableofcontents

\section{Introduction}

Generalized Igusa functions were introduced by Carnevale, Schein and Voll~\cite{CSV/19}.
Hall--Littlewood--Schubert series were introduced by Maglione and Voll~\cite{MaglioneVoll/24}.
Both are 
multivariate rational functions.
In each case, the authors proved a self-reciprocity result and provided many applications to enumeration problems in geometry, number theory, and algebra;
in particular, to explicit computations of zeta functions of groups, rings, and quiver representations.

The main purpose of the current work is to introduce a simultaneous generalization and refinement of these two rational functions, the \emph{skew Hall--Littlewood--Schubert series} $\HLS_{\bfn, \bfr}$ (see Definition~\ref{def:skewHLS}), and prove its self-reciprocity (Theorem~\ref{t:skewHLS_reciprocity}). 
In addition, our method of proof is elementary, avoiding $p$-adic integration which was used in~\cite{MaglioneVoll/24}.
The proof is modeled after a property of the M\"obius function of a partially ordered set, and was inspired by an idea of Sanyal.

Let us now describe briefly the main objects and results. More extensive definitions 
will be given in Section~\ref{sec:HLS_def}.
Let $\NN_0$ denote the set of non-negative integers.

\begin{defn}\label{def:Pnr_intro}
    Let $n, r \in \NN_0$.
    \begin{enumerate}[label=(\alph*)]
    
        \item
        Denote by $\Pt_{n,r}$ the set of all tuples $a = (a_0, a_1, \dots, a_n)$ of non-negative integers, where $a_0 \le r$, and $a_i \le 1$ for $1 \le i \le n$.
        Each element $a \in \Pt_{n,r}$ can be interpreted as the characteristic vector of a multiset containing at most $r$ copies of $0$, and at most one copy of each integer $1 \le i \le n$. In other words, it is a sub-multiset of $E_{n,r} \coloneqq \{0^r, 1, \dots, n\}$.

        \item
        For elements $a = (a_0, a_1, \dots, a_n)$ and $b = (b_0, b_1, \dots, b_n)$ in $\Pt_{n,r}$, write $a \le_t b$ if 
        \[
            \sum_{j=0}^{i} a_{j}
            \le \sum_{j=0}^{i}b_{j}
            \qquad (\forall\, 0\le i \le n).
        \]
    \end{enumerate}
\end{defn}

The relation $\letb$ is a partial order on $\Pt_{n,r}$, which we call the \emph{tableau order}. The choice of the name is explained in Section~\ref{sec:tableau-interpretation}.
The poset $(\Pt_{n,r}, \letb)$ has a unique minimal element $\varnothing$ and a unique maximal element $E_{n,r}$.

\begin{example}
    For $n = r = 2$, the partial order on the 12 elements of $\Pt_{2,2}$ is illustrated in Figure~\ref{fig:P22-Hasse}.
    Note that $(1,0,0)$ and $(0,1,1)$ are incomparable, and $(2,0,0)$ and $(1,1,1)$ are incomparable. We write sub-multisets without commas and parenthesis (e.g., $0^212$ for $E_{2,2}$).
    In terms of the corresponding multisets, 
    with elements written in non-decreasing order, 
    $0$ and $12$ are incomparable, and $0^2$ and $012$ are incomparable. 

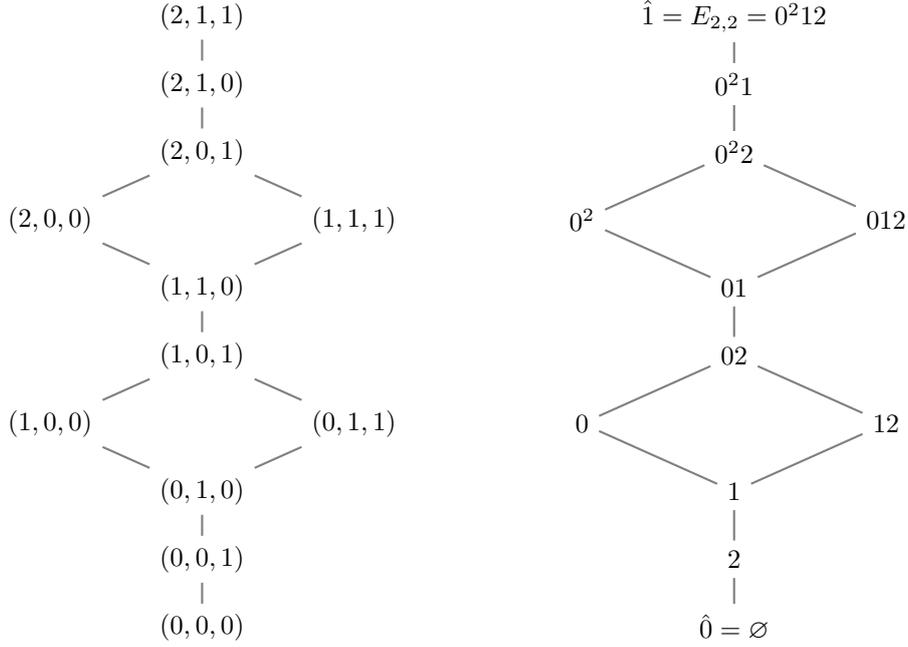
\begin{figure}
    \centering
    \begin{tikzpicture}[yscale=0.9, font=\small]
\begin{scope}
    \node at (0, 0) (000) {$(0,0,0)$};
    \node at (0, 1) (001) {$(0,0,1)$};
    \node at (0, 2) (010) {$(0,1,0)$};
    \node at (-2, 3) (100) {$(1,0,0)$};
    \node at (2, 3) (011) {$(0,1,1)$};
    \node at (0, 4) (101) {$(1,0,1)$};
    \node at (0, 5) (110) {$(1,1,0)$};
    \node at (-2, 6) (200) {$(2,0,0)$};
    \node at (2, 6) (111) {$(1,1,1)$};
    \node at (0, 7) (201) {$(2,0,1)$};
    \node at (0, 8) (210) {$(2,1,0)$};
    \node at (0, 9) (211) {$(2,1,1)$};
    \draw[color=gray,thick] (000) -- (001);
    \draw[color=gray,thick] (001) -- (010);
    \draw[color=gray,thick] (010) -- (100);
    \draw[color=gray,thick] (010) -- (011);
    \draw[color=gray,thick] (100) -- (101);
    \draw[color=gray,thick] (011) -- (101);
    \draw[color=gray,thick] (101) -- (110);
    \draw[color=gray,thick] (110) -- (200);
    \draw[color=gray,thick] (110) -- (111);
    \draw[color=gray,thick] (200) -- (201);
    \draw[color=gray,thick] (111) -- (201);
    \draw[color=gray,thick] (201) -- (210);
    \draw[color=gray,thick] (210) -- (211);
\end{scope}
\begin{scope}[xshift=7cm]
    \node at (0, 0) (000) {$\hat{0} = \varnothing$};
    \node at (0, 1) (001) {$2$};
    \node at (0, 2) (010) {$1$};
    \node at (-2, 3) (100) {$0$};
    \node at (2, 3) (011) {$12$};
    \node at (0, 4) (101) {$02$};
    \node at (0, 5) (110) {$01$};
    \node at (-2, 6) (200) {$0^2$};
    \node at (2, 6) (111) {$012$};
    \node at (0, 7) (201) {$0^22$};
    \node at (0, 8) (210) {$0^21$};
    \node at (0, 9) (211) {$\hat{1} = E_{2,2} = 0^212$};
    \draw[color=gray,thick] (000) -- (001);
    \draw[color=gray,thick] (001) -- (010);
    \draw[color=gray,thick] (010) -- (100);
    \draw[color=gray,thick] (010) -- (011);
    \draw[color=gray,thick] (100) -- (101);
    \draw[color=gray,thick] (011) -- (101);
    \draw[color=gray,thick] (101) -- (110);
    \draw[color=gray,thick] (110) -- (200);
    \draw[color=gray,thick] (110) -- (111);
    \draw[color=gray,thick] (200) -- (201);
    \draw[color=gray,thick] (111) -- (201);
    \draw[color=gray,thick] (201) -- (210);
    \draw[color=gray,thick] (210) -- (211);
\end{scope}
\end{tikzpicture}
    \caption{Hasse diagrams of $\Pt_{2,2}$. On the left  the elements are $3$-tuples, and on the right the elements are the corresponding sub-multisets of $E_{2,2}$.
    }
    \label{fig:P22-Hasse}
\end{figure}

\end{example}

More generally, fix a pair of $g$-tuples $\bfn, \bfr \in \NN_0^g$, say $\bfn = (n_1, \dots, n_g)$ and $\bfr = (r_1, \dots, r_g)$. Consider the direct product of posets
\begin{equation}\label{eq:Pbfnbfr}
    \Pt_{\bfn, \bfr} \coloneqq \Pt_{n_1, r_1} \times \dots \times \Pt_{n_g, r_g},
\end{equation}
where each component $\Pt_{n_i, r_i}$ has the tableau order $\letb$, and $\Pt_{\bfn, \bfr}$ has the product order:
\[
    (a^{(1)}, \dots, a^{(g)}) \le (b^{(1)}, \dots, b^{(g)})
    \ \iff\  (\forall i)\  a^{(i)} \letb b^{(i)}.
\]
We write $\letb$ also for the order in $\Pt_{\bfn, \bfr}$. Note that $\Pt_{\bfn, \bfr}$ has a unique minimal element $\hat{0} = (\varnothing, \dots, \varnothing)$ and a unique maximal element $\hat{1} = (E_{n_1, r_1}, \dots, E_{n_g, r_g})$.

We introduce $\prod_{i=1}^g (r_i+1)2^{n_i}$ variables~$\bfX = (X_c)$, where $c \in \Pt_{\bfn, \bfr}$, and $\sum_{i=1}^{g}(n_i + 1)$ variables $\bfY = (Y_{i, j})$, where $i \in [g] \coloneqq \{1, \dots, g\}$ and 
$0\le j \le n_i$.
For every chain $C$ in the half-open interval $(\hat{0},\hat{1}] \subset \Pt_{\bfn, \bfr}$, define a polynomial $W_C(\bfY) \in \ZZ[\bfY]$. The precise definition is quite intricate, and is deferred to Section~\ref{sec:HLS_def} (see Definition~\ref{def:W_poly}).

\begin{defn}\label{def:skewHLS_intro}
    The \emph{skew Hall--Littlewood--Schubert series} for $(\bfn, \bfr)$ is
    \[
    \HLS_{\bfn, \bfr}(\bfY; \bfX)
        \coloneqq
        \sum_{C} W_C(\bfY) \prod_{c \in C} \frac{X_c}{1-X_c} 
        \in \ZZ[\bfY](\bfX),
    \]
    where the summation is over all the strict chains $C$ in the interval $(\hat{0},\hat{1}] \subset \Pt_{\bfn, \bfr}$.
\end{defn}

The motivation behind Definition~\ref{def:skewHLS_intro} comes from two specializations, described in Section~\ref{sec:specializations}:
generalized Igusa functions~\cite[Definition~3.5]{CSV/19} 
and Hall--Littlewood--Schubert series~\cite[Definition~1.2]{MaglioneVoll/24}.
Definition~\ref{def:skewHLS_intro}, as a simultaneous generalization of these specializations, responds to a challenge 
posed in \cite[Remark~1.10]{MaglioneVoll/24}.

Write
$\bfY^{-1}$ for the set of inverses $Y_{i,j}^{-1}$ of the variables $Y_{i,j}$ in $\bfY$, and similarly $\bfX^{-1}$ for $\bfX$.
Our main result is the following.

\begin{theorem}[Skew $\HLS$ reciprocity]\label{t:skewHLS_reciprocity}
For any $\bfn, \bfr \in \NN_0^g$, the skew Hall--Littlewood--Schubert series satisfies
    \[
    \HLS_{\bfn, \bfr}(\bfY^{-1}; \bfX^{-1}) 
    = (-1)^{N}X_{\hat{1}}
    K(\bfY)^{-1}
    \cdot \HLS_{\bfn, \bfr}(\bfY;\bfX) ,
    \]
where $N=\sum_{i=1}^{g}(n_{i} + r_{i})$ and
\[
    K(\bfY)
    = %
    \prod_{i=1}^{g} \Bigg( Y_{i, 0}^{\binom{r_{i}}{2}} \prod_{j=1}^{n_i} Y_{i, j}^{r_i + j - 1} \Bigg).
\]
\end{theorem}

An outline of the rest of the paper is as follows: 
Section~\ref{sec:prelim} contains some preliminary definitions and notations.
Section~\ref{sec:HLS_def} presents detailed definitions of the main objects, the $\HLS$ and $\HLS'$ series, as well as their tableau interpretations.
It also states Theorem~\ref{t:skewModifiedHLS_reciprocity}, which is an $\HLS'$ version of Theorem~\ref{t:skewHLS_reciprocity}.
Section~\ref{sec:specializations} lists several specializations of the $\HLS$ series and its reciprocity property, which have appeared in various previous works and are now given a unified treatment. 
Section~\ref{sec:reciporcity-proof} contains a proof of Theorem~\ref{t:skewModifiedHLS_reciprocity}, using generalized zeta and M\"{o}bius matrices of a poset.
Finally, further comments and an open problem appear in Section~\ref{sec:further-discussion}.

\section{Preliminaries}
\label{sec:prelim}

Let $\NN$ be the set of positive integers, and let $\NN_0 = \{0\} \cup \NN$ be the set of non-negative integers.

\subsection{\texorpdfstring{$Y$}{Y}-analogs}

For $n,k \in \NN_0$ such that $k \le n$, and a variable $Y$, the \emph{$Y$\!-integer}, \emph{$Y$\!-factorial} and \emph{$Y$\!-binomial coefficient} are, respectively, 
\begin{align*}
    [n]_Y &\coloneqq \frac{1-Y^n}{1-Y}, &
    [n]_Y! &\coloneqq \prod_{i=1}^n [i]_Y, &
    & \text{and} &
    \binom{n}{k}_Y 
    &\coloneqq \frac{\left[n\right]_{Y}!}{\left[k\right]_{Y}!\left[n-k\right]_{Y}!}
    = \prod_{i=1}^{k}\frac{1-Y^{n-k+i}}{1-Y^{i}} .
\end{align*}
All of them are polynomials in $\NN_0[Y]$. 
Let $I$ be a finite multiset of integers in $[n]$, 
with elements $e_1 \le e_2 \le \dots \le e_\ell$.
Using $e_{\ell+1} \coloneqq n$, 
the corresponding \emph{$Y$\!-multinomial coefficient} is
\[
    \binom{n}{I}_Y \coloneqq \prod_{i=1}^{\ell} \binom{e_{i+1}}{e_{i}}_Y .
\]
Note that if $J$ is the underlying set of the multiset $I$, then $\binom{n}{I}_Y = \binom{n}{J}_Y$. We use the letter $Y$ instead of the more traditional $q$ since in applications, such as those in \cite{CSV/19,MaglioneVoll/24}, it is common for $q$ to denote a prime power and for the variables $\bfY$ and $\bfX$ to denote various (possibly negative) powers of $q$.

\subsection{Posets and chains}

Let $(P,\le)$ be a partially ordered set (poset).
For $a,b\in P$, we write $a<b$ if $a\le b$ and $a\ne b$. An \emph{open interval} $(a,b)$ is the set of all elements $c\in P$ such that $a<c<b$. 
A \emph{half-open interval} $(a,b]$ is $(a,b) \cup \{b\}$.

For any integer $k \ge 0$, a \emph{strict $k$-element chain ($k$-chain)} in a poset $P$ is an increasing sequence $c_1 < c_2 < \dots < c_k$ of elements of $P$. A \emph{non-strict $k$-element chain ($k$-multichain)} is 
a sequence $c_1 \le c_2 \le \dots \le c_k$.

\subsection{Partitions, diagrams and tableaux}

An (integer) \emph{partition} $\lambda$ is a non-increasing sequence $\lambda_1 \ge \lambda_2 \ge \dots \ge \lambda_k$ of positive integers, called the \emph{parts} of $\lambda$. 
The number $k \in \NN_0$ of parts is called the \emph{length} of $\lambda$, and for convenience we set $\lambda_i = 0$ for all $i > k$. The \emph{size} of $\lambda$ is $|\lambda| \coloneqq \lambda_1 + \dots + \lambda_k$.

To a partition $\lambda$ we associate a \emph{Young (or Ferrers) diagram}, which we denote by $[\lambda]$. It is a left-justified array of $|\lambda|$ identical square cells in the plane, arranged in $k$ rows, with $\lambda_i$ cells in the $i$-th row. According to the English convention, that we use throughout the paper, the first row (with $\lambda_1$ cells) is the top row. 
The partition $\lambda$ is the \emph{shape} of the diagram $[\lambda]$.

Let $\lambda_i'$ be the number of parts $\lambda_j$ such that $\lambda_j \ge i$. The \emph{conjugate partition} of $\lambda$, denoted by $\lambda'$, is the sequence $\lambda_1' \ge \lambda_2' \ge \dots \ge \lambda_{\lambda_1}'$. Equivalently, it is the partition whose diagram is the transpose of the diagram $[\lambda]$, obtained by reflecting along the main diagonal.

A \emph{semistandard Young tableau} of shape $\lambda$ is a filling of the cells of the Young diagram $[\lambda]$ with positive integers which are non-decreasing along rows (from left to right) and strictly increasing down columns.

A \emph{skew shape} is a pair $(\lambda,\mu)$ of partitions such that 
$\mu_i \le \lambda_i$ for all $i \ge 1$;
equivalently, the diagram $[\mu]$ is contained in the diagram $[\lambda]$.
This skew shape is denoted by $\lambda/\mu$.
The corresponding 
\emph{skew diagram} $[\lambda/\mu]$ is 
the set-theoretic difference of the diagrams $[\lambda]$ and $[\mu]$. 
A \emph{semistandard Young tableau} of skew shape $\lambda/\mu$ is a filling of the cells of the Young diagram $[\lambda]$ with integers such that
the cells of the sub-diagram $[\mu]$ are filled with zeros, and the cells of $[\lambda/\mu]$ are filled with positive integers which are non-decreasing along rows (from left to right) and strictly increasing down columns.
For a tableau $T$, denote by $T_{i,j}$ the entry in row $i$ and column $j$.

\begin{remark}\label{rem:zeros_in_SSYT}
    Our definition of semistandard tableaux of skew shape $\lambda/\mu$ is slightly non-standard, since we fill the cells of $[\mu]$ with zeros, thus recording the partition $\mu$ and not only the set-theoretic difference of the diagrams $[\lambda]$ and $[\mu]$. 
    This will be used in Definition~\ref{def:ThetaSkewTableau}.
    For example,
    if $\mu = \lambda$ we get, for distinct partitions $\lambda$, empty skew shapes $\lambda/\mu$ which yield distinct tableaux. 
\end{remark}

\section{Main definitions and their tableau interpretations}%
\label{sec:HLS_def}

In this section we introduce the $\HLS_{\bfn, \bfr}$ and the related $\HLS'_{\bfn, \bfr}$ series. Their definitions are formulated, in Subsections~\ref{sec:main-definition} and~\ref{sec:alternative-formulations}, in terms of chains in the poset $\Pt_{\bfn, \bfr}$. Later, in Subsection~\ref{sec:tableau-interpretation}, we give an equivalent formulation in terms of semistandard Young tableaux, which is closer in spirit to the notions introduced in \cite{MaglioneVoll/24}.

Subsection~\ref{sec:alternative-formulations} also contains Theorem~\ref{t:skewModifiedHLS_reciprocity}, which is a restatement of the main result, Theorem~\ref{t:skewHLS_reciprocity}, in terms of the $\HLS'_{\bfn, \bfr}$ series, and is the version to be proved in Section~\ref{sec:reciporcity-proof}.

\subsection{Main definition}
\label{sec:main-definition}

This subsection is dedicated to the definition of the skew Hall--Littlewood--Schubert series $\HLS_{\bfn, \bfr}$ (Definition~\ref{def:skewHLS}).
It is preceded by definitions of several auxiliary notations, serving to streamline the statement and proof of the main result, Theorem~\ref{t:skewHLS_reciprocity}.

\begin{defn}\label{def:s_and_Delta}
    For $a = (a_0, a_1, \dots, a_n) \in \Pt_{n,r}$ define 
    \begin{align*}
        s_0(a) \coloneqq \binom{a_0}{2} 
        \qquad \text{and} \qquad
        s_i(a) \coloneqq \sum_{k=0}^{i-1} a_k
        \quad (1 \le i \le n+1). 
    \end{align*}
    For $a, b \in \Pt_{n,r}$ define:
    \begin{align*}
        \Delta_i(a,b) &\coloneqq s_i(b) - s_i(a)
        \qquad (0 \le i \le n+1). 
    \end{align*}
\end{defn}

Thus $s_{n+1}(a)$ is the cardinality of $a$ (interpreted as a multiset) and, by definition,
\[
    a \letb b \iff
    (\forall\, 1 \le i \le n+1)\quad \Delta_i(a,b) \ge 0 .
\]

\begin{defn}
    Let $a, b \in \Pt_{n,r}$. For a variable $Y_0$, define
    \[
        \theta_{a,b}(Y_0) 
        \coloneqq \binom{b_0}{a_0}_{Y_0} ,
    \]
    where $a$ and $b$ contain $a_0$ and $b_0$ zeros, respectively.
\end{defn}

Note that $\theta_{a,b}(Y_0) = 0$ if $a_0 > b_0$, and that $\theta_{\varnothing,a}(Y_0) = \theta_{a,a}(Y_0) = 1$ while $\theta_{a,E_{n,r}}(Y_0) = \binom{r}{a_0}_{Y_0}$,
for every $a \in \Pt_{n,r}$.

\begin{defn}%
\label{def:SkewLegPairPoly}
    Set $n$ variables $\bfY_{>0} = (Y_1, \dots, Y_n)$.
    For $a, b \in \Pt_{n,r}$, define
    \[
        \Lplus_{a,b}
        \coloneqq \{i \in [n] \suchthat a_i = 1 \text{ and } b_i = 0\} .
    \]
    The \emph{refined leg polynomial} of the pair $(a,b)$ is
    \[
        \phi_{a,b}(\bfY_{>0})
        \coloneqq \begin{cases}
        \prod_{i \in \Lplus_{a,b}} \! \left( 1-Y_i^{\Delta_i(a,b)} \right)\!, & \text{if } a \letb b \:\!; \\
        0, & \text{otherwise.}
        \end{cases}
    \]
\end{defn}

\begin{observation}
    Assume that $a, b \in \Pt_{n,r}$ and $a \letb b$.
    \begin{enumerate}[label=(\alph*)]
        \item 
        If $\Lplus_{a,b} = \varnothing$, then $\phi_{a,b}(\bfY_{>0}) = 1$. This happens if and only if $a_i \le b_i$ for all $0 \le i \le n$; equivalently, $a$ is contained in $b$ in the multiset interpretation. This is similar to \cite[Remark~6.2]{MaglioneVoll/24}.
        \item 
        If $i \in \Lplus_{a,b}$, then $\Delta_i(a,b) > 0$.
        Therefore, if $\Lplus_{a,b} \ne \varnothing$, then $\phi_{a,b}(\bfY_{>0}) \ne 0, 1$.
    \end{enumerate}
\end{observation}

In particular,
    $\phi_{\varnothing,a}(\bfY_{>0}) 
    = \phi_{a,a}(\bfY_{>0}) 
    = \phi_{a,E_{n,r}}(\bfY_{>0}) 
    = 1$,
for every $a \in \Pt_{n,r}$.

\begin{example}
    $\phi_{247, 1456}(\bfY_{>0}) = (1 - Y_2)(1 - Y_7^2)$, $\phi_{247, 0^31456}(\bfY_{>0}) = (1 - Y_2^4)(1 - Y_7^5)$, $\phi_{146, 1456}(\bfY_{>0}) = 1$, and $\phi_{237, 0456}(\bfY_{>0}) = 0$.
\end{example}

Fix a pair of $g$-tuples $\bfn, \bfr \in \NN_0^g$, say $\bfn = (n_1, \dots, n_g)$ and $\bfr = (r_1, \dots, r_g)$, and recall from~\eqref{eq:Pbfnbfr} the poset
\[
    \Pt_{\bfn, \bfr} = \Pt_{n_1, r_1} \times \dots \times \Pt_{n_g, r_g}.
\]

\begin{defn}\label{def:w_ab_pair}
    Let $a, b \in \Pt_{\bfn, \bfr}$. Suppose $a = (a^{(1)}, \dots, a^{(g)})$ and $b = (b^{(1)}, \dots, b^{(g)})$ where $a^{(i)}, b^{(i)} \in \Pt_{n_i, r_i}$ for each $i \in [g]$. 
    Introduce variables $\bfY = (Y_{i, j})$, where $i \in [g]$ and $0\le j \le n_i$.
    For convenience, set $\bfY_i = (Y_{i, j})_{j \in [n_i]}$ for every $i \in [g]$.
    Define 
    \[
        w_{a,b}(\bfY) \coloneqq
        \prod_{i=1}^{g} \theta_{a^{(i)}, b^{(i)}}(Y_{i,0}) \phi_{a^{(i)}, b^{(i)}}(\bfY_i) .
    \]
\end{defn}

Note that $w_{\hat{0},a}(\bfY) = w_{a,a}(\bfY) = 1$
for every $a \in \Pt_{\bfn,\bfr}$.

\begin{defn}\label{def:W_poly}
    Let $C$ be any $k$-multichain $\hat{0} \letb c_1 \letb \dots \letb c_k \letb \hat{1}$ in 
    $\Pt_{\bfn, \bfr}$. 
    Denote also $c_0 \coloneqq \hat{0}$ and $c_{k+1} \coloneqq \hat{1}$. Define
    \[
        W_C(\bfY)
        = \prod_{j = 0}^{k} w_{c_j, c_{j+1}}(\bfY).
    \]
\end{defn}

We now define the main object of this paper,
using strict chains.
An alternative formulation using multichains will be given in Subsection~\ref{sec:alternative-formulations}.

Introduce variables~$\bfX = (X_c)$, for $c \in \Pt_{\bfn, \bfr}$.

\begin{defn}\label{def:skewHLS}
    The \emph{skew Hall--Littlewood--Schubert series} for $(\bfn, \bfr)$ is
    \[
    \HLS_{\bfn, \bfr}(\bfY; \bfX)
        \coloneqq
        \sum_{C} W_C(\bfY) \prod_{c \in C} \frac{X_c}{1-X_c} 
        \in \ZZ[\bfY](\bfX),
    \]
    where 
    the summation is over all the strict chains $C$ in the half-open interval $(\hat{0},\hat{1}] \subset \Pt_{\bfn, \bfr}$.
\end{defn}

\begin{example}
    Let $g=1$, $\bfn = (1)$ and $\bfr = (2)$. Use the simplified notations $Y_0 \coloneqq Y_{1,0}$ and $Y_1 \coloneqq Y_{1,1}$. With the aid of SageMath~\cite{SageMath/25}, one can compute that $\HLS_{(1), (2)}(\bfY; \bfX) = \Num / \Den$, where
    \begin{align*}
        \Num 
        &= Y_{0} Y_{1}^{2} X_{0^{2}} X_{01} X_{0} X_{1} + Y_{1}^{2} X_{0^{2}} X_{01} X_{1} + Y_{1}^{2} X_{0^{2}} X_{0} X_{1} - Y_{0} Y_{1} X_{0^{2}} X_{01} 
        - Y_{0} Y_{1} X_{0} X_{1} 
        &  \\
        &\quad 
        - Y_{1}^{2} X_{0^{2}} X_{1}
        - Y_{0} X_{01} X_{0} 
        - Y_{1} X_{0^{2}} X_{01} 
        - Y_{1} X_{0} X_{1} 
        + Y_{0} X_{01} 
        + Y_{0} X_{0} + 1
    \end{align*}
    and 
    \begin{align*}
        \Den
        &= (1 - X_{1}) (1 - X_{0}) (1 - X_{01}) (1 - X_{0^{2}}) (1 - X_{0^{2}1}) .
    \end{align*}
\end{example}

\begin{remark}
    The number of terms in the numerator (and denominator) of $\HLS_{\bfn, \bfr}$ grows rapidly with $\bfn$ and $\bfr$. For example, if $\bfn = (2)$ and $\bfr = (2)$, then $\HLS_{(2), (2)}(\bfY; \bfX)$ can be presented as $\Num / \Den$, where $\Num$ is an irreducible polynomial with $1{,}412$ terms and $\Den$ is the product $\prod_{\varnothing \ne a \in \Pt_{2,2}} (1 - X_c)$ with $11$ factors.
\end{remark}

\subsection{Two alternative formulations}
\label{sec:alternative-formulations}

Definition~\ref{def:skewHLS} can be restated in a very natural form, using summation over the (infinitely many) \emph{multichains} in $(\hat{0},\hat{1}]$.

\begin{defn}
\label{def:tcol-relation}
Let $C$ be a multichain in $\Pt_{\bfn, \bfr}$.
\begin{enumerate}[label=(\alph*)]
    \item
    For an element $c \in \Pt_{\bfn, \bfr}$, let $m_c(C)$ be the number of times $c$ appears in $C$.
    \item
    Let $\Mult(C) \coloneqq (m_c(C))_{c}$ be the \emph{multiplicity vector} of $C$.
    \item
    Let $\supp(C) \coloneqq \{c \suchthat m_c(C) \ne 0\}$.
    \item
    For a pair of multichains $C$ and $C'$, write $C \tcol C'$ if $\supp(C) = \supp(C')$.
\end{enumerate}
\end{defn}

The relation $\tcol$ is an equivalence relation on multichains. %
Clearly $C \tcol C'$ implies $W_C(\bfY) = W_{C'}(\bfY)$. Using the notation $\bfX^{\Mult(C)} \coloneqq \prod_c X_c^{m_c(C)}$, the $\HLS_{\bfn, \bfr}$ series can be presented as the formal power series
\begin{equation}
    \HLS_{\bfn, \bfr}(\bfY; \bfX)
        = \sum_{C} W_C(\bfY) \bfX^{\Mult(C)} ,\label{eq:skewHLS-gf}
\end{equation}
where the summation is over all the multichains $C$ in the half-open interval $(\hat{0},\hat{1}] \subset \Pt_{\bfn, \bfr}$.

We now
introduce a slight modification of 
$\HLS_{\bfn, \bfr}$,
which will be useful in the proof of the reciprocity property.

\begin{defn}\label{def:modified-skewHLS}
Let $\bfn, \bfr \in \NN_0^g$ be two $g$-tuples.
The \emph{modified skew Hall--Littlewood--Schubert series} for $(\bfn, \bfr)$ is
    \[
        \HLS'_{\bfn, \bfr}(\bfY; \bfX)
        \coloneqq
        \sum_{C} W_C(\bfY) \prod_{c \in C} \frac{X_c}{1-X_c} 
        \in \ZZ[\bfY](\bfX),
    \]
    where
    $W_C(\bfY)$ is as in Definition~\ref{def:W_poly}
    and the summation is over all the strict chains $C$ in the open interval $(\hat{0},\hat{1}) \subset \Pt_{\bfn, \bfr}$.
\end{defn}

\begin{remark}
    In Definition~\ref{def:modified-skewHLS} the summation is over chains in the \emph{open} interval $(\hat{0},\hat{1})$, whereas in Definition~\ref{def:skewHLS} it is over chains in the \emph{half-open} interval $(\hat{0},\hat{1}]$. This is the only difference between the two definitions.
\end{remark}

Note that every chain $C$ in $(\hat{0},\hat{1}]$ is either a chain in $(\hat{0},\hat{1})$, or is the chain $C'$ obtained from some chain $C$ in $(\hat{0},\hat{1})$ by appending the element $\hat{1}$. 
In the latter case, by Definition~\ref{def:W_poly}, $W_C(\bfY) = W_{C'}(\bfY)$.
It follows that
\[
    \HLS_{\bfn, \bfr}(\bfY; \bfX) 
    = \left( 1 + \frac{X_{\hat{1}}}{1 - X_{\hat{1}}} \right) \cdot \HLS'_{\bfn, \bfr}(\bfY; \bfX)  
    = \frac{1}{1 - X_{\hat{1}}} \cdot \HLS'_{\bfn, \bfr}(\bfY; \bfX) .
\]
Theorem~\ref{t:skewHLS_reciprocity} is therefore equivalent to the following modified reciprocity property.

\begin{theorem}[$\HLS'$ reciprocity]\label{t:skewModifiedHLS_reciprocity} 
For any $\bfn, \bfr \in \NN_0^g$,
\[
    \HLS'_{\bfn, \bfr}(\bfY^{-1}; \bfX^{-1})
    = (-1)^{N-1}
    K(\bfY)^{-1}
    \cdot \HLS'_{\bfn, \bfr}(\bfY; \bfX) ,
\]
where $N=\sum_{i=1}^{g}(n_{i} + r_{i})$ and
\[
    K(\bfY)
    = %
    \prod_{i=1}^{g} \Bigg( Y_{i, 0}^{\binom{r_{i}}{2}} \prod_{j=1}^{n_i} Y_{i, j}^{r_i + j - 1} \Bigg).
\]
\end{theorem}

The proof of this theorem appears in Section~\ref{sec:reciporcity-proof}.

\subsection{Tableau interpretation}
\label{sec:tableau-interpretation}

The definition of the $\HLS$ series by Maglione and Voll~\cite[Definition~1.2]{MaglioneVoll/24} uses semistandard Young tableaux (of regular shapes). 
In this subsection, we give a similar interpretation of the skew $\HLS$ series, $\HLS_{\bfn, \bfr}$, using semistandard Young tableaux of \emph{skew} shapes.
To this end, we provide tableau interpretations of the partial order $\letb$ and of the polynomials $\theta$, $\phi$, and $w$ from Section~\ref{sec:main-definition}.

Recall Definition~\ref{def:Pnr_intro}, where the poset $(\Pt_{n,r}, \letb)$ was introduced. 
For $r = 0$, this order is equivalent to the order used by Gale~\cite{Gale/68}.

\begin{remark}
    The inequalities defining $\letb$ are the same as those defining the dominance (or majorization) order on the partitions of a fixed number $n$. 
    The two orders are different, though, since the underlying sets are obviously different.
\end{remark}

We now interpret $\letb$ in terms of tableaux.
    
\begin{observation}
    Let $a, b \in \Pt_{n,r}$. Then $a \letb b$ if and only if the two-column tableau, composed of $a$ as the right column and of $b$ as the left column, is semistandard. 
\end{observation}

\begin{example}\label{exa:tableau-simple}
The following tableau
\ytableausetup{boxsize=1.3em}
\smallskip{}
    \[
    \begin{ytableau}
        0 & 0 & 0 & 0 & 0 & 2 & 2 \\
        0 & 0 & 0 & 1 & 3 & 5 \\
        2 & 5 \\
        3
    \end{ytableau}
    \]
is semistandard, of skew shape $(7,6,2,1) / (5,3)$; see Remark~\ref{rem:zeros_in_SSYT}. Its columns correspond to a strict $7$-element chain in $P_{5,2}$:
\[
    2 \lttb 25 \lttb 03 \lttb 01 \lttb 0^2 \lttb 0^2 5 \lttb 0^2 23 .
\]
\end{example}

We now interpret the polynomials $\theta$, $\phi$, and $w$ in terms of tableaux.
    
Let $\SSYT_{n,r}$ be the set of all semistandard Young tableaux $T = (T_{i,j})$ of some skew shape $\lambda / \mu$, 
such that the diagram $[\mu]$ has at most $r$ rows and is filled with zeros,
while the rest of the diagram $[\lambda]$ has entries from $[n]$; see Remark~\ref{rem:zeros_in_SSYT}.
We use the English convention; thus the entries of each $T \in \SSYT_{n,r}$ are non-decreasing along rows (from left to right) and down columns, with the positive entries of each column strictly increasing.
The columns of $T$ may be viewed as (not necessarily distinct) elements of $\Pt_{n,r}$.
The \emph{leg} of the cell at position $(i,j)$ is the multiset $\Leg_T(i,j) \coloneqq \{T_{k,j} \suchthat k\ge i\}$.
Informally, it is the sub-multiset of the $j$-th column of $T$ obtained by removing the smallest $i-1$ elements.

\begin{defn}\label{def:ThetaSkewTableau}
    Let $T\in \SSYT_{n,r}$ be a tableau of skew shape $\lambda / \mu$ with columns $c_{1},\dots,c_{\ell}$, listed from right to left.
    Suppose that $c_i$ contains $e_i$ zeros, i.e., $\mu$ has the conjugate integer partition $\mu'$ with parts $e_\ell \ge \dots \ge e_1$. For a variable $Y_0$, set
    \[
        \Theta_T(Y_0) 
        \coloneqq \prod_{i=1}^{\ell} \binom{e_{i+1}}{e_{i}}_{Y_0} 
        \in \ZZ[Y_0] ,
    \]
    a product of $Y_0$-binomial coefficients, where $e_{\ell+1} \coloneqq r$.
\end{defn}

In other words, the polynomial $\Theta_T(Y_0)$ is the $Y_0$-multinomial coefficient $\binom{r}{\mu'}_{Y_0}$. Note that if $r\le 1$, then $\Theta_T(Y_0) = 1$ for all $T \in \SSYT_{n,r}$.

\begin{defn}\label{def:SkewLegPoly}
    Set $n$ variables $\bfY_{>0} = (Y_1, \dots, Y_n)$, and let $T\in \SSYT_{n,r}$ be a semistandard Young tableau. 
    The \emph{refined leg polynomial} of $T$ is
    \[
        \Phi_{T}(\bfY_{>0})
        = \prod_{(i,j) \in \LTplus} \! \left( 1-Y_{T_{i, j+1}}^{\#\Leg_{T}^{+}(i,j)} \right) \ \in \ZZ[\bfY_{>0}] ,
    \]
    where
    \[
        \Leg_{T}^{+}(i,j)
        = \begin{cases}
        \{T_{k,j} \suchthat k\ge i \text{ and } T_{k,j} < T_{i, j+1}\}, & \text{if \ensuremath{T_{i, j+1}} exists and \ensuremath{T_{i, j+1}\notin \Leg_T(i,j)};}\\
        \varnothing, & \text{otherwise}
        \end{cases}
    \]
    is a sub-multiset of $\Leg_T(i,j)$ 
    of size $\#\Leg_{T}^{+}(i,j)$, 
    and $\LTplus = \left\{ (i,j)\in\NN^{2}\suchthat\Leg_{T}^{+}(i,j)\ne\varnothing\right\} $.
\end{defn}

\begin{example}
    Let $T$ be the tableau from Example~\ref{exa:tableau-simple}. Then
    \[
    \LTplus = \{(3,1), (2,3), (2,4), (1,5), (2,5)\}
    \]
    and
    \[
    \Phi_{T}(\bfY_{>0}) = (1 - Y_5^2)(1 - Y_1)(1 - Y_3)(1 - Y_2)(1 - Y_5).
    \]
\end{example}

Note that even though $0$ can appear as an entry in $T$, the polynomial $\Phi_{T}$ does not depend on the variable $Y_0$, since $T_{i, j+1} > 0$ for all $(i, j) \in \LTplus$. 
If $n=0$, then $\lambda/\mu$ is necessarily an empty skew shape, and $T$ amounts to filling the diagram $[\mu]$ with zeros, yielding $\Phi_{T}(\bfY_{>0}) = 1$. 
If $r=0$, then $E_{n,r}=[n]$ is
a set, and we recover the univariate leg polynomial \cite[Definition~1.1]{MaglioneVoll/24} upon substituting $Y_k = Y$ for all $k \in [n]$.

\begin{remark}
    Let $T\in \SSYT_{n,r}$ be a tableau with columns $c_{1},\dots,c_{\ell}$, listed from \emph{right to left}. Then, interpreting the columns as multisets,
    \[
    \Phi_{T}(\bfY_{>0}) = 1
    \iff \LTplus = \varnothing
    \iff (\forall j)\  c_j \subseteq c_{j+1}.
    \]
    This is similar to \cite[Remark~6.2]{MaglioneVoll/24}.
    Hence, by substituting $Y_k = 1$ for all $k \in [n]$, the leg polynomial $\Phi_{T}(\bfY_{>0})$ turns into an indicator function which has the value $1$ if the columns of $T$ form a multichain under multiset containment, and $0$ otherwise.
\end{remark}

\begin{defn}\label{def:projected-tableau}
    Let $C$ be a $k$-multichain $\hat{0} \lttb c_1 \letb \dots \letb c_k \letb \hat{1}$ in the half-open interval $(\hat{0}, \hat{1}] \subset \Pt_{\bfn, \bfr}$. For every $i \in [g]$ define a map $\pi_i$ that sends $C$ to a \emph{projected tableau} $\pi_i(C) \in \SSYT_{n_i, r_i}$ by considering only the $i$-th component in each element $c_j$ of $C$, and interpreting this chain of sub-multisets of $E_{n_i, r_i}$ as a (possibly empty) tableau.
\end{defn}

Notice that the columns of a projected tableau $\pi_i(C)$ 
are not necessarily distinct
when $g > 1$, even if $C$ is a strict chain.

\begin{example}
    For $g = 2$, let $\bfn = (4, 3)$ and $\bfr = (2, 3)$.  Consider the strict chain
    \begin{align*}
        C = (4, \varnothing) &\lttb (24, 2) \lttb (23, 23) \lttb (034, 023) \lttb (013, 023) \\
        &\lttb (013, 013) \lttb (0123, 0^223) \lttb (0^21234, 0^313)
    \end{align*}
    in $\Pt_{\bfn, \bfr}$. The two projected tableaux of $C$ are
    \ytableausetup{boxsize=1.3em}
    \begin{align*}
        & \begin{ytableau}
        0 & 0 & 0 & 0 & 0 & 2 & 2 & 4 \\
        0 & 1 & 1 & 1 & 3 & 3 & 4 \\
        1 & 2 & 3 & 3 & 4 \\
        2 & 3 \\
        3 \\
        4
        \end{ytableau} 
        && \begin{ytableau}
        0 & 0 & 0 & 0 & 0 & 2 & 2 \\
        0 & 0 & 1 & 2 & 2 & 3 \\
        0 & 2 & 3 & 3 & 3 \\
        1 & 3 \\
        3
        \end{ytableau} \\[1ex]
        \vspace{5ex}
        &\quad\ \ \pi_1(C) \in \SSYT_{4,2} &
        &\quad\ \ \pi_2(C) \in \SSYT_{3,3}
    \end{align*}
    and each of them has repeating columns.
\end{example}

\begin{observation}
    The polynomial from Definition~\ref{def:W_poly} can be written as
\[
    W_C(\bfY) 
    = \prod_{i=1}^{g} \Theta_{\pi_i(C)}(Y_{i,0}) \Phi_{\pi_i(C)}(\bfY_i) \in \ZZ[\bfY].
\]
\end{observation}

\begin{remark}\label{rem:refined_weights}
    Recall that the \emph{weight} of a tableau $T$ is the vector $\wt(T) = (m_i)_i$, where $m_i$ is the number of cells with entry $i$ in $T$. The generating function~\eqref{eq:skewHLS-gf}, under substitutions such as  $X_c = \prod_{i\in c} x_i$ for $g = 1$, belongs to the large family of generating functions that sum $\prod_i x_i^{m_i}$ over all semistandard tableaux of a given shape (yielding Schur functions), a given weight (e.g., only standard tableaux), or with a bounded maximal entry (as in our case).
    See \cite[(6.6) and Section~7]{MaglioneVoll/24}.
\end{remark}

\section{Specializations}
\label{sec:specializations}

This work was motivated by several reciprocity results, which are special cases of Theorem~\ref{t:skewHLS_reciprocity}. In this section we recall the corresponding specializations of $\HLS_{\bfn, \bfr}(\bfY; \bfX)$, state their reciprocity properties and make some additional remarks.

\subsection{The \texorpdfstring{$g = 1$}{g = 1} case}

In this case, $\HLS_{\bfn,\bfr} = \HLS_{(n),(r)}$ is a simultaneous generalization of the classical Igusa function and the Hall--Littlewood--Schubert series. 

On the one hand, if
$n=0$, the poset $\Pt_{0,r}$ is a chain with $r+1$ elements. The skew $\HLS$ series with these parameters is the following important function.

\begin{defn}
    Let $r \in \NN_0$.
    Using a single variable $Y$
    and $r$ variables $\bfX = (X_i)_{i \in [r]}$,
    the \emph{classical Igusa function} is
    \[
        I_r(Y; \bfX) 
        \coloneqq \sum_{J \subseteq [r]} \binom{r}{J}_Y \prod_{j\in J} \frac{X_j}{1 - X_j}.
    \]
\end{defn}

This function was introduced in~\cite[Theorem~4]{Voll/05}, and used there and in~\cite{ScheinVoll1/15}
to prove functional equations for zeta functions of nilpotent groups and rings. The proof of these functional equations relies on the following reciprocity property.

\begin{theorem}[{\cite[Proposition~4.2]{ScheinVoll1/15}}]
    For all $r\in \NN_0$,
    \[
    I_r(Y^{-1}; \bfX^{-1}) = (-1)^r X_r\, Y^{-\binom{r}{2}} I_r(Y; \bfX).
    \]
\end{theorem}

On the other hand, if $n \in \NN_0$ is arbitrary and $r=0$, then the skew $\HLS$ series is the Hall--Littlewood--Schubert series introduced in~\cite[Definition~1.2]{MaglioneVoll/24}, upon substituting $Y_k = Y$ for all $k \in [n]$.
Here, the poset $\Pt_{n,0} \setminus \{\varnothing\}$ is isomorphic to $\mathsf{T}_n$ from~\cite[Section~6]{MaglioneVoll/24}, but note that our partial order $\letb$ is opposite to the order $\sqsubseteq$ used there. Also, we define $\letb$ on all of $\Pt_{n,0}$. The inclusion of $\varnothing$ has the benefit that, for general $n$ and $r$, the graded poset $\Pt_{n,r}$ is self-dual, with the complementation map $f\colon \Pt_{n,r} \to \Pt_{n,r}$, defined by $f(a) \coloneqq E_{n,r} \setminus a$ for all $a \in \Pt_{n,r}$, serving as an order-reversing bijection.

For the original definition of the $\HLS$ series, write $\Phi_T(Y)$ for the univariate leg polynomial, by substituting $Y_k = Y$ for all $k \in [n]$ in Definition~\ref{def:SkewLegPoly}. 
A semistandard Young tableau is called \emph{reduced} if its columns are pairwise distinct; write $\rSSYT_{n,r}$ for the finite set of reduced tableaux in $\SSYT_{n,r}$. Set $\rSSYT_{n} \coloneqq \rSSYT_{n,0}$ for any $n$, and
consider a tableau $T \in \rSSYT_{n}$ as 
the set of its columns.

\begin{defn}[{\cite[Definition~1.2]{MaglioneVoll/24}}]
    Let $n\in \NN$, and let $\bfX = (X_J)_{\varnothing \ne J \subseteq [n]}$ be $2^n - 1$ variables. The \emph{Hall--Littlewood--Schubert series} is
    \[
    \HLS_n(Y, \bfX)
        \coloneqq
        \sum_{T \in \rSSYT_{n}} \Phi_T(Y) \prod_{c \in T} \frac{X_c}{1-X_c}.
    \]
\end{defn}

In \cite{MaglioneVoll/24}, several enumeration problems are solved by suitable substitutions for the variables of $\HLS_n(Y, \bfX)$.
We remark that the univariate leg polynomial $\Phi_T(Y)$ also appears as a special case of other combinatorial functions, especially in the study of Macdonald polynomials and Pieri rules; see, e.g., \cite[(VI.6.19), (VI.7.11')]{Macdonald/95} and 
\cite{ColmenarejoRam/24, KonvalinkaLauve/13, Warnaar/13}.

The reciprocity property of $\HLS_n$ is another corollary of Theorem~\ref{t:skewHLS_reciprocity}.

\begin{theorem}[{\cite[Theorem~A]{MaglioneVoll/24}}]
For all $n \in \NN$ we have
\[
    \HLS_n(Y^{-1}, \bfX^{-1}) 
    = (-1)^n X_{[n]} Y^{-\binom{n}{2}}\, \HLS_n(Y, \bfX).
\]
\end{theorem}

\subsection{The general \texorpdfstring{$g$}{g} case}
\label{sec:general-g-specializations}

Let $\bfn = \bfzero = (0, \dots, 0) \in \NN_0^g$, and let $\bfr \in \NN_0^g$ be arbitrary. The corresponding $\HLS_{\bfzero, \bfr}$ series is the generalized Igusa function. In particular, for $\bfr = \bfone = (1,\dots,1)  \in \NN_0^g$, the skew $\HLS_{\bfzero,\bfone}$ series is the weak order Igusa function.

\begin{defn}[{\cite[Definition~2.9]{ScheinVoll1/15}}]
    Let $g \in \NN$, and let $\bfX = (X_J)_{\varnothing \ne J \subseteq [g]}$ be $2^g - 1$ variables.
    The \emph{weak order Igusa function} is
    \[
        I_g^{\WO}(\bfX) 
        \coloneqq \sum_{\varnothing \ne J_{1} \subsetneq \dots \subsetneq J_{\ell} \subseteq [g]} \prod_{i=1}^{\ell} \frac{X_{J_{i}}}{1-X_{J_{i}}}.
    \]
\end{defn}

Note that none of the variables $\bfY$ from the skew $\HLS$ series appears in this expression, since the $Y$-multinomial coefficients $\Theta_T(Y) = 1$ are all trivial.
This function was introduced in~\cite{ScheinVoll1/15}, and used for computations of the local normal zeta functions of Heisenberg groups over rings of integers of number fields, for unramified primes. Its reciprocity property \cite[Proposition~4.1]{ScheinVoll1/15} is a corollary of Theorem~\ref{t:skewHLS_reciprocity}.

More generally, if $\bfn = \bfzero \in \NN_0^g$ and $\bfr = (r_1, \dots, r_g) \in \NN_0^g$ is arbitrary, 
then
    \[
        \HLS_{\bfzero,\bfr}(\bfY; \bfX)
        = I^\WO_{\bfr}(\bfY; \bfX)
    \]
is the \emph{generalized Igusa function}~\cite[Definition~3.5]{CSV/19}, 
having applications to
zeta functions of a large class of nilpotent groups and rings.
Here the variables $\bfY$ are $(Y_{i,0})_{i \in [g]}$, and the poset $\Pt_{\bfzero,\bfr}$ is a direct product of chains.
For every multichain $C$ in $(\hat{0}, \hat{1}]$ and index $i \in [g]$,
the projected tableau $\pi_i(C)$ is filled with zeros, 
hence $\Phi_{\pi_i(C)}(\bfY_i) = 1$.
The generalized Igusa function
interpolates between two special cases already mentioned:
the classical Igusa function, corresponding to $\bfr = (r)$, and the weak order Igusa function, corresponding to $\bfr = \bfone$.

The following reciprocity property of the generalized Igusa function is a corollary of Theorem~\ref{t:skewHLS_reciprocity}.

\begin{theorem}[{\cite[Theorem~3.8]{CSV/19}}]
    Let $\bfr = (r_1, \dots, r_g) \in \NN_0^g$ 
    and $N \coloneqq \sum_i r_i$. Then
    \[
        I_{\bfr}^{\WO}(\bfY^{-1}; \bfX^{-1}) 
        = (-1)^{N}X_{\hat{1}} %
        \left(\prod_{i=1}^{g}Y_{i}^{-\binom{r_{i}}{2}}\right)I_{\bfr}^{\WO}(\bfY; \bfX).
    \]
\end{theorem}

We are not aware of any previous appearance in the literature of the series $\HLS_{\bfn,\bfzero}$ with $g>1$ and $\bfn$ including parts greater than $1$.

\begin{remark}
    It is clear that the weak order Igusa function is a special case of the generalized Igusa function.
    However, in \cite[Proposition~8.5]{MaglioneVoll/24} it is shown that $I_g^{\WO}(\bfX)$ is also the specialization $\HLS_g(1, \bfX)$ of the original Hall--Littlewood--Schubert series.
    This is not a coincidence, since the posets $\Pt_{0,1}$ and $\Pt_{1,0}$ are isomorphic, and so $\Pt_{\bfzero,\bfone} \cong \Pt_{\bfone,\bfzero}$. More generally, we have $\Pt_{n,1} \cong \Pt_{n+1,0}$ for any $n \in \NN_0$.
\end{remark}

\section{Proof of reciprocity}
\label{sec:reciporcity-proof}

The proof of self-reciprocity in~\cite{MaglioneVoll/24} uses $p$-adic integration.
In this section we give a combinatorial proof of Theorem~\ref{t:skewModifiedHLS_reciprocity}, 
which is equivalent to Theorem~\ref{t:skewHLS_reciprocity}.
The proof follows an idea of R.~Sanyal
(private communication), who 
used it to prove self-reciprocity for generalized Igusa functions via combinatorial reciprocity \cite{BeckSanyal/18}.

We start with Lemma~\ref{t:unitriangular_inv} and Lemma~\ref{t:P_alternating_sum}, which are general facts.
The technical heart of this paper is Lemma~\ref{t:refined_Z_and_M}.
It is followed by Lemma~\ref{t:refined_gen_Z_and_M}, which is an extension for any $g \ge 1$.
Lemma~\ref{t:refined_gen_Z_and_M} is the main tool used in the proof of Theorem~\ref{t:skewModifiedHLS_reciprocity}.

\subsection{Zeta and M\"obius matrices}

The following result is folklore; we provide a short proof for completeness.

\begin{lemma}\label{t:unitriangular_inv}
    If $U = (u_{i,j})$ is a strictly upper-triangular square matrix over a commutative ring with $1$, then $I + U$ is invertible. Its inverse is of the form $I + V$ for a strictly upper-triangular matrix $V = (v_{i,j})$. Moreover, 
    \[
        v_{i,j}
        = \sum_{k = 0}^{j-i-1} (-1)^{k+1} 
        \sum_C
        \,\prod_{t = 0}^{k} u_{c_t,c_{t+1}}
        \qquad (\forall\, i < j),
    \] %
    where the inner summation is over all strict $k$-element chains $C = \{c_1 < \dots < c_k\}$ in the open interval $(i,j)$, with $c_0 \coloneqq i$ and $c_{k+1} \coloneqq j$.
\end{lemma}

\begin{proof}
    If $U$ is strictly upper-triangular of size $n \times n$, then $U^n = 0$ and thus
    \[
        I = I^n - (-U)^n = (I + U) \cdot \sum_{k=0}^{n-1}(-U)^k .
    \]
    Similarly, 
    \[
        \sum_{k=0}^{n-1}(-U)^k \cdot (I + U) = I .
    \]
    It follows that the matrix $I+U$ is invertible, with unique inverse $I+V$, where
    \[
        V 
        = \sum_{k=1}^{n-1}(-U)^k 
        = \sum_{k=0}^{n-2}(-U)^{k+1} .
    \]
    This matrix $V$ is clearly strictly upper-triangular.
    For any $i < j$, the entry in position $(i,j)$ of $U^{k+1}$ is 
    \[
        U_{i,j}^{k+1} 
        = \sum_k \sum_C \prod_{t = 0}^{k} u_{c_t,c_{t+1}} ,
    \]
    where the summations are over all integers $k \ge 0$ and over all strict $k$-element chains $C = {\{c_1 < \dots < c_k\}}$ in the open interval $(i,j)$, with $c_0 \coloneqq i$ and $c_{k+1} \coloneqq j$. Substituting this into the last expression for $V$ completes the proof.
\end{proof}

The following 
result
generalizes a classical expression of the M\"obius function of a poset in terms of its zeta function. This explains our choice of the letters $Z$ and $M$ to denote the matrices.

\begin{lemma}\label{t:P_alternating_sum}
    Let $(P, \le_P)$ be a finite poset, and let $\le_L$ be an arbitrary linear order on the set $P$. Let $Z = (z_{a,b})$ be a square matrix over a commutative ring with $1$, with rows and columns indexed by the elements of $P$ according to the linear order $\le_L$. If
    \begin{align*}
        z_{a,b} &= 0
        \qquad (\forall\, a \not\le_P b), \\
        z_{a,a} &= 1
        \qquad (\forall\, a \in P) ,
    \end{align*}
    and $z_{a,b}$ arbitrary for $a <_P b$,
    then $Z$ is invertible, and its inverse $M = (m_{a,b})$ satisfies 
    \begin{align*}
        m_{a,b} &= 0
        \qquad (\forall\, a \not\le_P b), \\
        m_{a,a} &= 1
        \qquad (\forall\, a \in P) ,
    \end{align*}
    and
    \[
        m_{a,b} 
        = \sum_{k \ge 0} (-1)^{k+1} 
        \sum_C
        \,\prod_{t = 0}^{k} z_{c_t,c_{t+1}}
        \qquad (\forall\, a <_P b) ,
    \]
    where the inner summation is over all strict $k$-element chains $C = \{c_1 <_P \dots <_P c_k\}$ in the open interval $(a,b) \subset P$, with $c_0 \coloneqq a$ and $c_{k+1} \coloneqq b$.
\end{lemma}

\begin{proof}
    The three formulas defining $m_{a,b}$ depend on the partial order $\le_P$, but not on the chosen linear order $\le_L$.
    Replacing this linear order by any other linear order on the set $P$ amounts to replacing $Z$ and $M$ by $Q Z Q^{-1}$ and $Q M Q^{-1}$, respectively, for a suitable permutation matrix $Q$. If $M$ is the (left and right) inverse of $Z$ then $Q M Q^{-1}$ is the inverse of $Q Z Q^{-1}$. Therefore, it suffices to prove the claim for one specific linear order. 
    
    Let $\le_L$ be a linear extension of $\le_P$. If $a >_L b$, namely $a \not\le_L b$, then $a \not\le_P b$ and thus, by assumption, $z_{a,b} = 0$. Moreover, $z_{a,a} = 1$ for all $a \in P$. It follows that $Z = I + U$ for some strictly upper-triangular matrix $U$. By Lemma~\ref{t:unitriangular_inv}, the matrix $Z$ is invertible, with inverse matrix of the form $M = I + V$ for a strictly upper-triangular matrix $V$, and satisfies
    \[
        m_{a,b} 
        = \sum_{k \ge 0} (-1)^{k+1} 
        \sum_C
        \,\prod_{t = 0}^{k} z_{c_t,c_{t+1}}
        \qquad (\forall\, a <_L b) ,
    \]
    where the inner summation is over all strict $k$-element chains $C = \{c_1 <_L \dots <_L c_k\}$ in the open interval $(a,b) \subset P$, with $c_0 \coloneqq a$ and $c_{k+1} \coloneqq b$.
    The chains $C$ are with respect to the linear order $<_L$, but note that $c_t \ne c_{t+1}$ for all $0 \le t \le k$, and therefore $z_{c_t,c_{t+1}} \ne 0$ implies $c_t <_P c_{t+1}$ by the assumption on $Z$. It follows that nonzero products can only arise when $a = c_0 <_P c_1 <_P \dots <_P c_k <_P c_{k+1} = b$, and in particular $a <_P b$. This completes the proof.
\end{proof}

\subsection{Zeta and M\"obius matrices over \texorpdfstring{$\Pt_{n,r}$}{P\_{n,r}}}

Recall Definition~\ref{def:w_ab_pair} of the polynomial $w_{a,b}(\bfY)$. In the following Lemma~\ref{t:refined_Z_and_M} we consider only the case $g = 1$, and therefore use simplified notation: $\bfY = (Y_0, Y_1, \dots, Y_n)$ for the variables, as well as $\bfY_{>0} = (Y_1, \dots, Y_n)$ and
\[
    w_{a,b}(\bfY) = \theta_{a,b}(Y_0) \phi_{a,b}(\bfY_{>0}) \in \ZZ[\bfY]
\]
for every $a, b\in \Pt_{n,r}$.
Recall also the notation $\Delta_i(a,b)$ from Definition~\ref{def:s_and_Delta}.

\begin{lemma}\label{t:refined_Z_and_M}
    Fix a linear order $\le_L$ on $\Pt_{n,r}$.
    Let $Z = (z_{a,b})$ and $M = (m_{a,b})$ be matrices of size $(r+1)2^n \times (r+1)2^n$, with rows and columns indexed by the elements of $\Pt_{n,r}$ according to the order $\le_L$, and with the following entries:
    \begin{align*}
        z_{a,b} 
        &= w_{a,b}(\bfY) \\
        m_{a,b} 
        &= (-1)^{\Delta_{n+1}(a,b)} \prod_{i=0}^{n} Y_i^{\Delta_i(a,b)} 
        \cdot w_{a,b}(\bfY^{-1})
    \end{align*}
    for all $a, b \in \Pt_{n,r}$. Then $M$ is the inverse of $Z$.
\end{lemma}

\begin{proof}
    The proof is by induction on $n$.
    Assume first that $n = 0$ (and $r \ge 0$ is arbitrary).

    For $a = (a_0)$ and $b = (b_0)$ in $\Pt_{0,r}$, we have $a \letb b$ if and only if $a_0 \le b_0$. 
    Clearly 
    \[
        \phi_{a,b}(\bfY_{>0}) 
        = \begin{cases}
            1, & \text{if } a_0 \le b_0; \\
            0, & \text{otherwise.}
        \end{cases}
    \]
    It follows that, for all $a = (a_0)$ and $b = (b_0)$ in $\Pt_{0,r}$,
    \begin{align*}
        z_{a,b} &= \binom{b_0}{a_0}_{Y_0} \\
        m_{a,b} &= (-1)^{\Delta_1(a,b)} Y_0^{\Delta_0(a,b)} \binom{b_0}{a_0}_{Y_0^{-1}} .
    \end{align*}
    For convenience, change notation slightly from $a_0$, $b_0$, and $Y_0$ to $i$, $j$, and $q$, respectively.
    Using the natural linear order $0 \le_L 1 \le_L \cdots \le_L r$, the entries of the matrices are
    \begin{align*}
        z_{i,j} &= \binom{j}{i}_{q} \\
        m_{i,j} &= (-1)^{j - i} q^{\binom{j}{2} - \binom{i}{2}} \binom{j}{i}_{q^{-1}} .
    \end{align*}
    
    The matrix $Z$ is known as the $q$-Pascal matrix, and its inverse and 
    other properties
    are well-known;
    see, e.g., Comtet~\cite[p.~118]{Comtet/74}, Zheng~\cite[equation~(2.5)]{Zheng/08}, or Verde-Star~\cite[equation~(4.17)]{VerdeStar/11}.
    For completeness, we include a short proof that $M$ is the inverse of $Z$ in this case.

    Working over the field of rational functions $\QQ(q)$, 
    it suffices to show that $ZM = I$.
    Element-wise, we want to show, for all $0 \le i,k \le r$, that
    \[
        \sum_{j} z_{ij} m_{jk} = \delta_{ik} \coloneqq
        \begin{cases}
            1, \qquad i = k\\
            0, \qquad i \ne k.
        \end{cases}
    \]
    Since $Z$ and $M$ are upper-triangular, it suffices to prove these equalities for $i \le k$. 
    Indeed, noting that $[i]!_{q^{-1}} = q^{-\binom{i}{2}} [i]!_q$ and $\binom{j}{i}_{q^{-1}} = q^{\binom{i}{2} + \binom{j-i}{2} -\binom{j}{2}} \binom{j}{i}_{q}$, we have for $i \le k$
    \begin{align*}
        \sum_{j} z_{ij} m_{jk} 
        &= \sum_{j=i}^{k} \binom{j}{i}_{q} (-1)^{k - j} q^{\binom{k}{2} - \binom{j}{2}} \binom{k} {j}_{q^{-1}} 
        = \sum_{j=i}^{k} \binom{j}{i}_{q} (-1)^{k - j} q^{\binom{k-j}{2}} \binom{k}{j}_{q} \\
        &= \sum_{j=i}^{k} (-1)^{k - j} q^{\binom{k-j}{2}} \binom{k}{i}_{q} \binom{k-i}{k-j}_{q} 
        = \binom{k}{i}_{q} \cdot \sum_{t=0}^{k-i} (-1)^{t} q^{\binom{t}{2}} \binom{k-i}{t}_{q} = \delta_{ik},
    \end{align*}
    where the last equality follows from taking $x = 1$, $y = -1$, and $N = k - i$ in the $q$-binomial theorem
    \[
        \prod_{t=0}^{N-1} (x + q^t y) 
        = \sum_{t=0}^{N} \binom{N}{t}_q q^{\binom{t}{2}} x^{N-t} y^t .
    \]
    This completes the proof for $n = 0$.

    For the induction step (on $n$), fix any $r \ge 0$. 
    We want to use a linear order on $\Pt_{n,r}$ for which the matrices $Z$ and $M$ have a nice block structure, but are not necessarily upper triangular (namely, the linear order is not necessarily a linear extension of $\letb$). We choose the \emph{reverse lexicographic} order $\lerl$, defined as follows: 
    if $a = (a_0, a_1, \dots, a_n)$ and $b = (b_0, b_1, \dots, b_n)$ are in $\Pt_{n,r}$, and $a \ne b$, then 
    $a \ltrl b$ if and only if there exists $0 \le k \le n$ such that $a_i = b_i$ for all $i > k$ while $a_k < b_k$.
    Viewed as multisets, $a \ltrl b$ if the maximal element of the symmetric difference of $a$ and $b$ belongs to $b$.
    
    For example, if $n = r = 2$ then, in the multiset interpretation of the elements of $\Pt_{2,2}$,
        \[
            \varnothing \lerl 0 \lerl 00 \lerl 1 \lerl 01 \lerl 001 \lerl 2 \lerl 02 \lerl 002 \lerl 12 \lerl 012 \lerl 0012. 
        \]
    Note that the first $(r+1)2^n$ elements of $P_{n+1,r}^t$ in this order are the elements of $\Pt_{n,r}$ (in this order), and they are followed by the same elements with $n+1$ added. 
    
    Denote by $Z$ and $M$ the matrices corresponding to $\Pt_{n,r}$ (with rows and columns in reverse lexicographic order), and by $Z^+$ and $M^+$ the matrices corresponding to $\Pt_{n+1,r}$, considered as $2\times 2$ block matrices with blocks of the same size. Their entries are in $\ZZ[\bfY]$ where $\bfY = (Y_0, \dots, Y_{n+1})$.
    For a multiset $a \in \Pt_{n,r}$ denote $a^+ \coloneqq a \cup\{n+1\}$.
    The four blocks of $Z^+$ (and $M^+$) correspond to the four types of row-column pairs:
    \[
        (a,b), \quad 
        (a,b^+), \quad 
        (a^+,b), \quad
        (a^+,b^+) 
        \qquad (\forall\, a, b \in \Pt_{n,r}).
    \]
    
    Observe that
    \begin{align*}
        a \letb b^+ &\iff a \letb b , \\
        a^+ \letb b &\iff a \letb b \text{\ and\ } \Delta_{n+2}(a,b) \ge 1 , \\
        a^+ \letb b^+ &\iff a \letb b.
    \end{align*}
    Elementary computation gives $\Delta_{n+2}(a,b) = \Delta_{n+1}(a,b)$, and so
    \begin{align*}
        \Delta_{n+2}(a,b^+) &= \Delta_{n+1}(a,b) + 1 , \\
        \Delta_{n+2}(a^+,b) &= \Delta_{n+1}(a,b) - 1 , \\
        \Delta_{n+2}(a^+,b^+) &= \Delta_{n+1}(a,b) ;
    \end{align*}
    While for all $0 \le i \le n+1$:
    \[
        \Delta_i(a,b)
        = \Delta_i(a,b^+) 
        = \Delta_i(a^+,b)
        = \Delta_i(a^+,b^+).
    \]
    Also,
    \[
        \theta_{a,b}(Y_0)
        = \theta_{a,b^+}(Y_0)
        = \theta_{a^+,b}(Y_0)
        = \theta_{a^+,b^+}(Y_0).  
    \]
    Computing the refined leg polynomials gives
    \begin{align*}
        \phi_{a,b^+}(\bfY_{>0}) &= \phi_{a,b}(\bfY_{>0}) \\
        \phi_{a^+,b^+}(\bfY_{>0}) &= \phi_{a,b}(\bfY_{>0}).
    \end{align*}
    The computation of $\phi_{a^+,b}(\bfY_{>0})$ is a little more involved. 
    If $a \letb b$ but $\Delta_{n+1}(a,b) = 0$, then $a^+ \not\letb b$ hence $\phi_{a^+,b}(\bfY_{>0}) = 0$. 
    If $a \letb b$ and $\Delta_{n+1}(a,b) \ge 1$, then $a^+ \letb b$, so the set $\LTplus$ from Definition~\ref{def:SkewLegPoly} for the tableau $T$ with columns $b$ and $a^+$ has one more element than the corresponding set for the tableau with columns $b$ and $a$.
    Hence
    \[
        \phi_{a^+,b}(\bfY_{>0}) 
        = \left( 1 - Y_{n+1}^{\Delta_{n+1}(a,b)} \right) \cdot \phi_{a,b}(\bfY_{>0}). 
    \]
    Interestingly, this formula is also valid if $\Delta_{n+1}(a,b) = 0$. The above inductive description is related to \cite[Lemma~4.4]{MaglioneVoll/24}.
    
    Combining all of these facts, we obtain for $Z^+$:
    \begin{align*}
        z^+_{a,b} 
        &= w_{a,b}(\bfY) 
        = z_{a,b} \,, \\
        z^+_{a,b^+} 
        &= w_{a,b^+}(\bfY) 
        = w_{a,b}(\bfY) 
        = z_{a,b} \,, \\
        z^+_{a^+,b} 
        &= w_{a^+,b}(\bfY) 
        =  \left( 1 - Y_{n+1}^{\Delta_{n+1}(a,b)} \right) \cdot w_{a,b}(\bfY) 
        = \left( 1 - Y_{n+1}^{\Delta_{n+1}(a,b)} \right) \cdot z_{a,b} \,, \\
        z^+_{a^+,b^+} 
        &= w_{a^+,b^+}(\bfY) 
        = w_{a,b}(\bfY) 
        = z_{a,b} \,.
    \end{align*}
    Similarly, for $M^+$:
    \begin{align*}
        m^+_{a,b} 
        &= (-1)^{\Delta_{n+2}(a,b)} \prod_{i=0}^{n+1} Y_i^{\Delta_i(a,b)} 
        \cdot w_{a,b}(\bfY^{-1}) 
        = Y_{n+1}^{\Delta_{n+1}(a,b)} \cdot m_{a,b} \,, \\
        m^+_{a,b^+} 
        &= (-1)^{\Delta_{n+2}(a,b^+)} \prod_{i=0}^{n+1} Y_i^{\Delta_i(a,b^+)} 
        \cdot w_{a,b^+}(\bfY^{-1}) \\
        &= -Y_{n+1}^{\Delta_{n+1}(a,b)} \cdot (-1)^{\Delta_{n+1}(a,b)} \prod_{i=0}^{n} Y_i^{\Delta_i(a,b)} 
        \cdot w_{a,b}(\bfY^{-1}) \\
        &= -Y_{n+1}^{\Delta_{n+1}(a,b)} \cdot m_{a,b} \,, \\
        m^+_{a^+,b^+} 
        &= (-1)^{\Delta_{n+2}(a^+,b^+)} \prod_{i=0}^{n+1} Y_i^{\Delta_i(a^+,b^+)} 
        \cdot w_{a^+,b^+}(\bfY^{-1}) \\
        &= Y_{n+1}^{\Delta_{n+1}(a,b)} \cdot m_{a,b} \,,
    \end{align*}
    and the slightly more involved entry
    \begin{align*}
        m^+_{a^+,b}
        &= (-1)^{\Delta_{n+2}(a^+,b)} \prod_{i=0}^{n+1} Y_i^{\Delta_i(a^+,b)} \cdot w_{a^+,b}(\bfY^{-1}) \\
        &= (-1)^{\Delta_{n+1}(a,b)} \cdot (- Y_{n+1}^{\Delta_{n+1}(a,b)}) \prod_{i=0}^{n} Y_i^{\Delta_i(a,b)} \cdot \left( 1 - Y_{n+1}^{-\Delta_{n+1}(a,b)} \right) \cdot w_{a,b}(\bfY^{-1}) \\  
        &= - Y_{n+1}^{\Delta_{n+1}(a,b)} \left( 1 - Y_{n+1}^{-\Delta_{n+1}(a,b)} \right) \cdot m_{a,b} 
        = \left( 1 - Y_{n+1}^{\Delta_{n+1}(a,b)} \right) \cdot m_{a,b} \,.
    \end{align*}
    Since $\Delta_{n+1}(a,b) = s_{n+1}(b) - s_{n+1}(a)$, the quantity $Y_{n+1}^{\Delta_{n+1}(a,b)} z_{a,b}$ is the entry in row $a$ and column $b$ of the matrix $D^{-1} Z D$, where $D = (d_{a,b})$ is a diagonal matrix with $d_{a,a} = Y_{n+1}^{s_{n+1}(a)}$ for all $a \in \Pt_{n,r}$.
    We can therefore write
    \[
        Z^+
        = \begin{pmatrix*}[c]
            Z & Z \\ 
            Z - D^{-1} Z D & Z
        \end{pmatrix*}
    \]
    and
    \[
        M^+
        = \begin{pmatrix*}[c]
            D^{-1} M D & -D^{-1} M D \\ 
            M - D^{-1} M D & \hphantom{-}D^{-1} M D
        \end{pmatrix*} .
    \]
    Multiplication, using the induction hypothesis $Z M = I$, gives
    \[
        Z^+ M^+
        = \begin{pmatrix*}[c]
            Z M & 0 \\ 
            -D^{-1} Z M D + Z M & D^{-1} Z M D
        \end{pmatrix*}
        = \begin{pmatrix*}[c]
            I & 0 \\ 
            0 & I
        \end{pmatrix*} ,
    \]
    as required.
\end{proof}

We can derive a similar formula for any $g \ge 1$.

\begin{lemma}\label{t:refined_gen_Z_and_M}
    For general $g \ge 1$, fix a linear order $\le_L$ on $\Pt_{\bfn,\bfr}$.
    Let $Z = (z_{a,b})$ and $M = (m_{a,b})$ be matrices of size $\prod_{i=1}^{g}(r_i+1)2^{n_i} \times \prod_{i=1}^{g}(r_i+1)2^{n_i}$, with rows and columns indexed by the elements of $\Pt_{\bfn,\bfr}$ according to the order $\le_L$, and with the following entries:
    \begin{align*}
        z_{a,b} 
        &= w_{a,b}(\bfY) \\
        m_{a,b} 
        &= \prod_{i=1}^{g} \Bigg( (-1)^{\Delta_{n_i+1}(a^{(i)},b^{(i)})} \prod_{j=0}^{n_i} Y_{i,j}^{\Delta_j(a^{(i)},b^{(i)})} \Bigg)
        \cdot w_{a,b}(\bfY^{-1})
    \end{align*}
    for all $a, b \in \Pt_{\bfn,\bfr}$, with notation as in Definition~\ref{def:w_ab_pair}; in particular, $a = (a^{(1)}, \dots, a^{(g)})$ and $b = (b^{(1)}, \dots, b^{(g)})$, where $a^{(i)}, b^{(i)} \in \Pt_{n_i, r_i}$ for each $i \in [g]$.
    
    Then $M$ is the inverse of $Z$.
\end{lemma}

\begin{proof}
    Recall that the Kronecker product of an $n \times n$ matrix $U$ and an $n' \times n'$ matrix $V$ is the $nn' \times nn'$ matrix $U\otimes V$ defined as an $n \times n$ block matrix by
\[
    U\otimes V = \begin{pmatrix}
    u_{11}V & \cdots & u_{1n}V \\
     \vdots & \ddots & \vdots \\
    u_{n1}V & \cdots & u_{nn}V
    \end{pmatrix}.
\]
It is well known that $U\otimes V$ is invertible if and only if both $U$ and $V$ are invertible, and moreover in that case
$(U \otimes V)^{-1} = U^{-1} \otimes V^{-1}$.

The matrices $Z$ and $M$ have the structure of a Kronecker product:
\begin{align*}
    Z &= \bigotimes_{i=1}^{g} Z^{(i)} \,, &
    M &= \bigotimes_{i=1}^{g} M^{(i)},
\end{align*}
where $Z^{(i)}$ and $M^{(i)}$ are the $(r_i+1)2^{n_i} \times (r_i+1)2^{n_i}$ matrices defined in Lemma~\ref{t:refined_Z_and_M} for the poset $\Pt_{n_i, r_i}$. Thus,
\[
Z^{-1} = \Big( \bigotimes_i Z^{(i)} \Big)^{-1} = \bigotimes_i M^{(i)} = M. \qedhere
\]
\end{proof}

Informally, the matices $Z$ and $M$ in the previous lemma encode $\bfY$\!-analogs of the zeta and M\"obius functions of $(\Pt_{\bfn, \bfr}, \letb)$. This is a refinement of the $q$-analog of the zeta function of a poset recently introduced by Chapoton~\cite{Chapoton/24}.

\subsection{Proof of 
\texorpdfstring{$\HLS'$}{HLS'} reciprocity}

We are now ready to prove the reciprocity property of $\HLS'$, the modified $\HLS$ series. 
To that aim we first rewrite Definition~\ref{def:modified-skewHLS} of $\HLS'$, using Definition~\ref{def:W_poly}, with explicit attention to the length $k$ of each chain.

\begin{observation}\label{t:skewHLS'_successive}
    \[
        \HLS'_{\bfn, \bfr}(\bfY; \bfX) 
        = \sum_{k,\, C}
        \left(\, \prod_{j = 0}^{k} w_{c_j, c_{j+1}}(\bfY) \cdot \prod_{j = 1}^{k} \frac{X_{c_j}}{1-X_{c_j}} \right)
        \in \ZZ[\bfY](\bfX),
    \]
    where the summation is over all $k \ge 0$ and all strict $k$-element chains $C = {\{ c_1 \lttb \dots \lttb c_k \}}$ in the open interval $(\hat{0},\hat{1}) \subset \Pt_{\bfn, \bfr}$, with
    $c_0 = \hat{0}$ and $c_{k+1} = \hat{1}$.
\end{observation}

\begin{proof}[Proof of Theorem~\ref{t:skewModifiedHLS_reciprocity}]
    By Observation~\ref{t:skewHLS'_successive},
    \begin{equation}\label{eq:modifiedHLS'_inv}
        \HLS'_{\bfn, \bfr}(\bfY^{-1}; \bfX^{-1}) 
        = \sum_{k,\, C}
        \left(\, \prod_{j = 0}^{k} w_{c_j, c_{j+1}}(\bfY^{-1}) \cdot \prod_{j = 1}^{k} \frac{X_{c_j}^{-1}}{1-X_{c_j}^{-1}} \right) ,
    \end{equation}
    where the summation is over all $k \ge 0$ and all strict $k$-element chains $C = \{ c_1 \lttb \dots \lttb c_k \}$ in the open interval $(\hat{0},\hat{1}) \subset \Pt_{\bfn, \bfr}$, with $c_0 = \hat{0}$ and $c_{k+1} = \hat{1}$. 
    Now
    \[
        \frac{X_{c_j}^{-1}}{1 - X_{c_j}^{-1}}
        = \frac{1}{X_{c_j} - 1}
        = \frac{-1}{1 - X_{c_j}}
        = - \left( 1 + \frac{X_{c_j}}{1 - X_{c_j}} \right) ,
    \]
    so that
    \[
        \prod_{j = 1}^{k} \frac{X_{c_j}^{-1}}{1 - X_{c_j}^{-1}} 
        = (-1)^k \sum_{D \subseteq C} \prod_{s = 1}^{\ell} \frac{X_{d_s}}{1 - X_{d_s}},
    \]
    where the summation is over all subsets $D = \{d_1, \ldots, d_\ell\}$ of $C = \{c_1, \ldots, c_k\}$, namely over all $\ell \ge 0$ and all $\ell$-element chains $D = (\hat{0} = d_0 \lttb d_1 \lttb \cdots \lttb d_\ell \lttb d_{\ell+1} = \hat{1})$ which are coarser than the chain $C$.
    Substituting this into equation~\eqref{eq:modifiedHLS'_inv} and changing the order of summation, we get
    \begin{align*}
        \HLS'_{\bfn, \bfr}(\bfY^{-1}; \bfX^{-1}) 
        &= \sum_{k,\, C}
        \left( \prod_{j = 0}^{k} w_{c_j, c_{j+1}}(\bfY^{-1}) \cdot 
        (-1)^k \sum_{\ell,\, D \subseteq C} \prod_{s = 1}^{\ell} \frac{X_{d_s}}{1 - X_{d_s}} \right) \\ 
        &= \sum_{\ell,\, D}
        \left( \prod_{s = 1}^{\ell} \frac{X_{d_s}}{1 - X_{d_s}} \cdot
        \sum_{k,\, C \supseteq D} (-1)^k 
        \,\prod_{j = 0}^{k} w_{c_j, c_{j+1}}(\bfY^{-1}) \right) .
    \end{align*}
    The subset $D$ partitions $C \setminus D$ into $\ell + 1$ disjoint chains $C^{(0)}, \ldots, C^{(\ell)}$. If for each $0 \le s \le \ell$ the chain $C^{(s)}$ is $d_s \lttb c_{s,1} \lttb \dots \lttb c_{s,k_s} \lttb d_{s+1}$, of size $k_s$, then clearly $\sum_s k_s = k - \ell$.
    We also set $c_{s,0} \coloneqq d_s$ and $c_{s,k_s + 1} \coloneqq d_{s+1}$ for every $0 \le s \le \ell$.
    Therefore, we can express $\HLS'_{\bfn, \bfr}(\bfY^{-1}; \bfX^{-1})$ as
    \[
        \sum_{\ell,\, D} 
        \left( \prod_{s = 1}^{\ell} \frac{X_{d_s}}{1 - X_{d_s}}
        \cdot (-1)^\ell \prod_{s = 0}^{\ell}
        \,\sum_{k_s,\, C^{(s)} \subseteq (d_s, d_{s+1})} (-1)^{k_s}
        \,\prod_{j = 0}^{k_s} w_{c_{s,j}, c_{s,j+1}}(\bfY^{-1}) \right) .
    \]
    
    Recall the matrix $Z = (z_{a,b}) = (w_{a,b}(\bfY))$ and its inverse $M = (m_{a,b})$ from Lemma~\ref{t:refined_gen_Z_and_M}, with rows and columns indexed by the elements of $\Pt_{\bfn,\bfr}$. The entry of $M$ in row $d_s$ and column $d_{s+1}$ is
    \[
        m_{d_s,d_{s+1}} 
        = \prod_{i=1}^{g} \left( (-1)^{\Delta_{n_i+1}(d^{(i)}_s, d^{(i)}_{s+1})} \prod_{j=0}^{n_i} Y_{i,j}^{\Delta_j(d^{(i)}_s, d^{(i)}_{s+1})} \right)
        \cdot w_{d_s, d_{s+1}}(\bfY^{-1}) .
    \]
    On the other hand, by Lemma~\ref{t:P_alternating_sum} this entry is
    \[
        m_{d_s,d_{s+1}} 
        = \sum_{k_s,\, C^{(s)} \subseteq (d_s, d_{s+1})} (-1)^{k_s+1}
        \,\prod_{j = 0}^{k_s} w_{c_{s,j},c_{s,j+1}}(\bfY).
    \]
    Replacing $\bfY$ by $\bfY^{-1}$, the equality of these two expressions becomes
    \begin{multline*}
        (-1) \cdot \sum_{k_s,\, C^{(s)} \subseteq (d_s, d_{s+1})} (-1)^{k_s}
        \,\prod_{j = 0}^{k_s} w_{c_{s,j},c_{s,j+1}}(\bfY^{-1}) \\
        = \prod_{i=1}^{g} \left( (-1)^{\Delta_{n_i+1}(d^{(i)}_s, d^{(i)}_{s+1})} \prod_{j=0}^{n_i} Y_{i,j}^{-\Delta_j(d^{(i)}_s, d^{(i)}_{s+1})} \right)
        \cdot w_{d_s, d_{s+1}}(\bfY)\,.
    \end{multline*}
    It follows that $\HLS'_{\bfn, \bfr}(\bfY^{-1}; \bfX^{-1})$ is

    \begin{align*}
        &\quad \sum_{\ell,\, D} 
        \Bigg( \prod_{s = 1}^{\ell} \frac{X_{d_s}}{1 - X_{d_s}} \cdot 
        (-1) \cdot \prod_{s = 0}^{\ell} %
        \prod_{i=1}^{g} \Bigg(\! (-1)^{\Delta_{n_i+1}(d^{(i)}_s, d^{(i)}_{s+1})} \prod_{j=0}^{n_i} Y_{i,j}^{-\Delta_j(d^{(i)}_s, d^{(i)}_{s+1})} \Bigg) %
        w_{d_s, d_{s+1}}(\bfY) \!\Bigg) \\
        &= \sum_{\ell,\, D} 
        \!\Bigg( \prod_{s = 1}^{\ell} \frac{X_{d_s}}{1 - X_{d_s}} \cdot 
        (-1) \cdot 
        \prod_{i=1}^{g} \Bigg(\! (-1)^{\Delta_{n_i+1}(d^{(i)}_0,d^{(i)}_{\ell+1})} \prod_{j=0}^{n_i} Y_{i,j}^{-\Delta_j(d^{(i)}_0,d^{(i)}_{\ell+1})} \Bigg)
        \cdot \prod_{s = 0}^{\ell} w_{d_s, d_{s+1}}(\bfY) \!\Bigg) \!.
    \end{align*}
    By Definition~\ref{def:s_and_Delta}, for each $1\le i \le g$ and $1\le j \le n_i+1$:
    \[
        \Delta_{j}(d^{(i)}_0,d^{(i)}_{\ell+1}) 
        = \Delta_{j}(\varnothing, E_{n_i, r_i}) = r_i + j - 1,
    \]
    while 
    \[
        \Delta_{0}(d^{(i)}_0,d^{(i)}_{\ell+1}) 
        = \Delta_{0}(\varnothing, E_{n_i, r_i}) = \binom{r_i}{2}.
    \]
    Thus $\HLS'_{\bfn, \bfr}(\bfY^{-1}; \bfX^{-1})$ is
    \begin{align*}
        &\quad\, (-1) \cdot \prod_{i=1}^{g} \Bigg(\! (-1)^{r_i + n_i} Y_{i, 0}^{-\binom{r_{i}}{2}} \prod_{j=1}^{n_i} Y_{i, j}^{-(r_i + j - 1)} \Bigg) \cdot \sum_{\ell,\, D} 
        \left( \prod_{s = 1}^{\ell} \frac{X_{d_s}}{1 - X_{d_s}}
        \cdot \prod_{s = 0}^{\ell} w_{d_s, d_{s+1}}(\bfY) \right) \\
        &= (-1)^{-1 + \sum_{i=1}^{g}(n_{i} + r_{i})}
        \prod_{i=1}^{g} \Bigg( Y_{i, 0}^{-\binom{r_{i}}{2}} \prod_{j=1}^{n_i} Y_{i, j}^{-(r_i + j - 1)} \Bigg)
        \cdot \HLS'_{\bfn, \bfr}(\bfY; \bfX) ,
    \end{align*}
    which completes the proof.
\end{proof}

\section{Further discussion}
\label{sec:further-discussion}

Definition~\ref{def:skewHLS} of $\HLS_{\bfn, \bfr}$ is reminiscent of the flag Hilbert--Poincar\'e series of a hyperplane arrangement \cite{MVhyper/24}, and the motivic zeta function of a matroid \cite{JensenKutlerUsatine/21}.
Theorem~\ref{t:skewModifiedHLS_reciprocity} has an equivalent formulation in the spirit of \cite[Corollary~3.3]{MVhyper/24}. %
For each $S \subseteq (\hat{0}, \hat{1}) = \Pt_{\bfn, \bfr} \setminus \{\hat{0}, \hat{1}\}$, 
let $\OC(S)$ be the order (or chain) complex of the poset $S$. The elements of $\OC(S)$ are the strict chains in $S$. 
By multiplying the formula in Theorem~\ref{t:skewModifiedHLS_reciprocity} by $\prod_{a \in (\hat{0}, \hat{1})}(1-X_a)$ and comparing the coefficients of monomials $\prod_{a\in S} X_a$ for each $S$, the reciprocity property of the $\HLS'_{\bfn, \bfr}$ series is seen to be equivalent to the following statement.

\begin{corollary}\label{cor:alternative-S}
    For any $S \subseteq (\hat{0}, \hat{1})$
    denote $\Sc \coloneqq (\hat{0}, \hat{1}) \setminus S$. Then
    \[
        \sum_{C \in \OC(S)} (-1)^{|C|} W_C(\bfY) 
        = (-1)^{N-1} K(\bfY) \cdot \!\sum_{C \in \OC(\Sc)} (-1)^{|C|} W_C(\bfY^{-1}),
    \]
    where 
    $N=\sum_{i=1}^{g}(n_{i} + r_{i})$ and 
    \[
        K(\bfY)
        = %
        \prod_{i=1}^{g} \Bigg( Y_{i, 0}^{\binom{r_{i}}{2}} \prod_{j=1}^{n_i} Y_{i, j}^{r_i + j - 1} \Bigg).
    \]
\end{corollary}

This corollary is also reminiscent of
combinatorial Alexander duality for Eulerian posets; see, e.g., \cite[Theorem 1.1]{BjornerTancer/09} and \cite[Proposition~2.2]{Stanley/82}. Note that $\Pt_{\bfn, \bfr}$ is, in general, not Eulerian. It will be interesting to find another proof of the reciprocity via such route.

There is a natural generalization of the poset $(\Pt_{n,r}, \letb)$ introduced in Definition~\ref{def:Pnr_intro}.
Let $A = (m_0, m_1, \dots, m_n) \in \NN_0^{n+1}$ be a tuple of non-negative integers. Consider the poset $(P_A, \letb)$ whose elements are all the tuples $(a_0, a_1, \dots, a_n) \in \NN_0^{n+1}$ where $a_i \le m_i$ for all $0 \le i \le n$, with the same partial order $\letb$. 
Equivalently, the elements of $P_A$ can be interpreted as sub-multisets of $\{0^{m_0}, 1^{m_1}, \dots, n^{m_n}\}$. The case studied in this paper is $A = (r, 1, \dots, 1)$.

\begin{problem}\label{prob:HLS-generalization}
    Find a generalization, to any $A$, of the skew Hall--Littlewood--Schubert series, namely of the polynomials $W_C(\bfY)$, that satisfies a self-reciprocity property.
\end{problem}

\printbibliography

\end{document}